\documentclass[11pt]{amsart}
\usepackage{graphicx}
\allowdisplaybreaks
\usepackage{amsmath,amssymb,mathrsfs,amsthm}

\usepackage{verbatim}
\usepackage[spanish,USenglish]{babel} 
\usepackage{color}
\usepackage{tikz}
\usepackage{graphicx}

\newtheorem{thm}{Theorem}[section]
\newtheorem{corollary}[thm]{Corollary}
\newtheorem{proposition}[thm]{Proposition}
\newtheorem{theorem}[thm]{Theorem}
\newtheorem{lemma}[thm]{Lemma}

\theoremstyle{definition}
\newtheorem{definition}[thm]{Definition}
\newtheorem{example}[thm]{Example}
\theoremstyle{remark}
\newtheorem{remark}[thm]{Remark}

\newtheorem{question}[thm]{Question}
\numberwithin{equation}{section}

\def\N{\mathbb{N}}
\def\R{\mathbb{R}}
\def\D{\mathbb{D}}

\def\C{\mathbb{C}}

\def\Tn2{{\mathcal T}^{(n)}_2}
\def\LL{\mathcal{L}}

\def\NN{\mathbb{N}}
\def\RR{\mathbb{R}}

\def\CC{\mathbb{C}}

\def\CCma{\mathbb{C}^{+}}

\def\LL{\mathcal{L}}

\def\La{\textrm{L}}
\def\Po{\textrm{P}}

\def\ffi{\varphi}

\def\zc{\overline{z}}

\def\gwn{g_{\omega,n}}

\def\TniiRma{\mathcal{T}^{(n)}_{2}}

\def\HniiCma{H^{(n)}_{2}}

\def\HiinCma{{H}_{2}^{(n)}}

\begin{document}
\title[Hilbertian Hardy-Sobolev spaces]{Hilbertian Hardy-Sobolev spaces on a half-plane}

\author[Gal\'{e}]{Jos\'{e} E. Gal\'{e}}
\address{Departamento de Matem\'aticas, Instituto Universitario de Matem\'aticas y Aplicaciones, Universidad de Zaragoza, 50009 Zaragoza, Spain.}
\email{gale@unizar.es}

\author[Matache]{Valentin Matache}
\address{ Department of Mathematics, University of Nebraska\\
Omaha, NE 68182, USA.} \email{vmatache@unomaha.edu}

\author[Miana]{Pedro J. Miana}
\address{Departamento de Matem\'aticas, Instituto Universitario de Matem\'aticas y Aplicaciones, Universidad de Zaragoza, 50009 Zaragoza, Spain.}
\email{pjmiana@unizar.es}

\author[S\'{a}nchez-Lajusticia]{Luis S\'{a}nchez--Lajusticia}
\address{Departamento de Matem\'aticas, Facultad de Ciencias, Universidad de Zaragoza, 50009 Zaragoza, Spain.}
\email{luiss@unizar.es}

\thanks{Jos\'{e} E. Gal\'{e} and Pedro J. Miana   have been partially supported by Project MTM2016-77710-P, DGI-FEDER, of the MCYTS and Project E26-17R, D.G. Arag\'on, Spain.}



\maketitle

\begin{abstract} In this paper we deal with a scale of reproducing kernel Hilbert spaces  $H^{(n)}_2$, $n\ge 0$,
which are linear subspaces of the classical Hilbertian Hardy space on the right-hand half-plane $\C^+$.
They are  obtained as ranges of the Laplace transform in extended versions of the Paley-Wiener theorem
which involve absolutely continuous functions of higher degree.
An explicit integral formula is given for the reproducing kernel $K_{z,n}$ of $H^{(n)}_2$, from which we can find
the estimate
$\Vert K_{z,n}\Vert\sim\vert z\vert^{-1/2}$ for $z\in\C^+$. Then
composition operators $C_\varphi :H_2^{(n)} \to H_2^{(n)}$, $C_\varphi f=f\circ \varphi $, on these spaces are discussed, giving some necessary and some sufficient  conditions for analytic maps $\varphi: \C^+\to \C^+$  to induce bounded composition operators.
\end{abstract}

\section{Introduction}\label{intro}

As is well known, $C_0$-semigroups provide solutions to abstract Cauchy equations,
$$
\left\{
\begin{array}{ll}
u'(t)=Au(t),\quad& t\ge 0,\\
u(0)=x, \quad& x\in D(A),
\end{array}
\right.
$$
whenever such equations are well-posed
\cite{ABHN}. There still are other important Cauchy equations for which
$C_0$-semigroups do not work and that must be treated using a variety of methods or theories.
Among such theories, one finds in the literature those concerning
integrated semigroups or distribution semigroups.
At the intersection of these two settings, one has
the space ${\mathcal T}_1^{(n)}$ of functions on $\R^+:=(0,\infty)$ defined as the completion of the test function space
$C_c^\infty(\R^+)$ in the norm
$$
\Vert f\Vert_{1,(n)}:=\int_0^\infty\vert f^{(n)}(t)\vert t^n\ dt<\infty, \quad f\in C_c^\infty(\R^+).
$$
We have the continuous inclusions
$ {\mathcal T}_1^{(m)}\hookrightarrow{\mathcal T}_1^{(n)}\hookrightarrow {\mathcal T}_1^{(0)}=L^1(\R^+)$
for every $m\ge n$. Moreover, every ${\mathcal T}_1^{(n)}$ is a convolution Banach algebra for
the norm $\Vert \cdot\Vert_{1,(n)}$.
Being primarily  introduced in \cite{AK} to approach Cauchy problems (with a different notation), the space
${\mathcal T}_1^{(n)}$
has not been studied as a Banach algebra in itself until quite recently; as a sample,
see \cite{Galeetal}, \cite{GaleSanchez} and references therein. The character space of the  Banach algebra
${\mathcal T}_1^{(n)}$ is the closure, in the usual (Euclidean) topology, of the right-hand half-plane $\C^+:=\{z\in\C:\Re z>0\}$ and the associated Gelfand mapping coincides with the Laplace transform $\LL$. In \cite{Galeetal}, the range of the mapping $\LL$ restricted on
${\mathcal T}_1^{(n)}$ is studied, showing that it is dense in a certain algebra of holomorphic functions with continuous derivatives up to $i\R\setminus\{0\}$.

The space ${\mathcal T}_1^{(n)}$ is  a particular case of spaces ${\mathcal T}_p^{(n)}$ defined for
$1\le p <\infty$ by $g\in{\mathcal T}_p^{(n)}$ if and only if $g$ is a measurable function on $\R^+$
which belongs to the closure of $C_c^\infty(\R^+)$ in the norm
$$
\Vert f\Vert_{(n),p}:=\left(\int_0^\infty\vert f^{(n)}(t) t^n\vert^p\ dt\right)^{1/p}<\infty, \quad f\in C_c^\infty(\R^+).
$$

As in the case of ${\mathcal T}_1^{(n)}$ and in analogy with spaces $L^p(\R^+)$, $1\le p\le2$, it makes sense to study the range of the
Laplace transform acting on
${\mathcal T}_p^{(n)}$. Naturally, the first case to be considered is $p=2$. In \cite {Roy},
the range $\LL({\mathcal T}_2^{(n)})$ is characterized as the (Hilbert) space $H_2^{(n)}$ of holomorphic functions in $H_2$ such
that $z^k F^{(k)}\in  H_2$ for every $k=0,1,\dots, n$. Here, we  denote by $H_2$ the Hardy space of all analytic functions $F$ over $\C^+$ satisfying condition
\begin{equation}\label{Norm01}
\| F\| _2:=\sup_{x>0} \left(\frac{1}{2\pi }\int_{-\infty }^{\infty }|F(x+iy)|^2\, dy\right)^{1\over 2} <\infty.
\end{equation}
In fact, $H_2$ is a Hilbert space with the norm given by (\ref{Norm01}).

Point evaluations are continuous on the Hardy space $H_2$ so that it  is a reproducing kernel space (RKHS for short)
of holomorphic functions. In fact, its reproducing kernel $K(z,w)=K_w(z)$ and corresponding norm are given by
$$
K_w(z):={1\over (z+\overline{w})}, \hbox{ and } {\Vert K_w\Vert_2= {1\over\sqrt{2 (\Re w)}}},\quad z,w\in \C^+.
$$
For more details on the space $H_2$, the reader is referred to \cite[Chapter VI]{Duren} and \cite[Chapter 8]{Hoffman}; see also  \cite{JM, PR}.

Endowed with the norm $\Vert F\Vert_{2,(n)}:=\Vert z^nF\Vert_2$,
the space $H_2^{(n)}$ is contained in
$H_2$ continuously and so it is a RKHS. One of the main purposes in the present paper is to determine
the reproducing kernel of $K_n(z,w)=K_{n,w}(z)$ for $H_2^{(n)}$ and to estimate the norm $\Vert K_{n,w}\Vert_{2,(n)}$.
In effect, for every $n\in\N$, we find such a kernel in the form
$$
K_{n,w}(z)=\frac{1}{((n-1)!)^2}\int _0^1\int _0^1\frac{(1-t)^{n-1}(1-s)^{n-1}}{zt+s\overline{w}}\, ds\, dt, \quad z, w\in \C ^+,
$$
which can be regarded as an integrated version of the kernel $K(z,w)$ described above. Then a detailed analysis of that integral
gives us  the estimate
$$
\frac{1}{(n-1)!\sqrt{{ (2n-1)}}}\ {1\over\sqrt{|w|}}\leq \|K_{n,w}\|_{2,(n)}\leq \frac{\sqrt{\pi}}{(n-1)!\sqrt{{ n}}}\ {1\over\sqrt{|w|}},
\quad w\in \C^+.
$$

A notable difference between norms $\Vert K_w\Vert_2$ and $\|K_{n,w}\|_{2,(n)}$ deserves attention.
Whereas the norm $\Vert K_w\Vert_2$ is equivalent to the inverse of the distance of point $w$ to
the boundary of $\C^+$, as is the case of some other RKHS of analytic functions on their domains, the norm
$\|K_{n,w}\|_{2,(n)}$ depends on the {\it radial} distance of $w$ to the origin in the complex plane.

The estimate of the norm $\|K_{n,w}\|_{2,(n)}$ given above suggests studying
composition operators
$C_\varphi$
on $H_2^{(n)}$,  that is, $C_\varphi\colon H_2^{(n)}\to H_2^{(n)}$ defined by
$C_\varphi(f):=f\circ \varphi$ for all $f\in H_2^{(n)}$ and suitable analytic selfmaps $\varphi$ of $\C^+$. Here $\varphi$ is said to be the symbol inducing $C_\varphi$.

The motivation for such a search comes naturally from the well known Caughran--Schwartz identity, which
applied to the kernel function $K_{n,w}$ on $\C^+$ reads
$
C^*_\varphi (K_{n,w})=K_{n,\varphi (w)}$ for $w\in \C^+,$ where $C_\varphi$ is bounded and $C^*_\varphi$ is its adjoint (\cite{PR}). This identity readily implies
$$\sup_{w\in\C^+} \{\Vert K_{n,\varphi (w)}\Vert_{2,(n)}  \Vert K_{n,w}\Vert_{2,(n)}^{-1}\}\leq \| C_\varphi \|,$$ from which we get the necessary condition
\begin{equation}\label{radialsup}
\sup_{w\in \C^+} \frac{\vert w\vert }{\vert\varphi (w)\vert}<\infty,
\end{equation}
clearly related to the angular derivative $\lim_{z\to\infty}\vert w\vert\vert\varphi (w)\vert^{-1}$.

Thus the finiteness property given by (\ref{radialsup}) is a necessary condition for the boundedness of $C_\varphi$.
Conversely, we will show that the following more restrictive
condition implies the boundedness of $C_\varphi$:
\begin{equation}\label{Derivasup}
\sup_{z\in \C^+}\left|  z^k\frac{\varphi ^{(k)}(z)}{\varphi (z)}\right|  <\infty,\quad k=1,2,3, \dots, n.
\end{equation}
We do not know whether any of conditions (\ref{radialsup}) or (\ref{Derivasup}) is both necessary and sufficient for the boundedness of  $C_\varphi$.

\medskip
The paper is organized as follows.
In Section \ref{SobolevLebesgue}, we define spaces ${\mathcal T}_2^{(n)}$ --and their corresponding inner products--
such as they are introduced in \cite{Roy} (in a slightly different but equivalent way from above). Then we provide
${\mathcal T}_2^{(n)}$
with another, equivalent, inner product --which allows us to intertwine $n$-times derivation
and multiplication by the power $t^n$--,
and give a relevant example of function in ${\mathcal T}_2^{(n)}$. Such an example will be used in Section
\ref{RKHSobolev} to find the kernel $K_{n}$ which generates ${\mathcal T}_2^{(n)}$. The Paley-Wiener type theorem saying that
$\LL({\mathcal T}_2^{(n)})=H_2^{(n)}$ isometrically was first proved in \cite{Roy}. In Section \ref{HardySobolev}, a different, shorter,
proof of this result is given on the basis of the equivalent inner product quoted above and the usage of Laguerre and Legendre polynomials.  We give the connection of spaces $H_2^{(n)}$
with subspaces of the usual Hardy space $H_2(\D)$ in
the unit disc $\D$, obtained via the Cayley transform between the disc $\D$ and the half-plane $\C^+$ in Section \ref{HardySobolev2}.
As said before, Section \ref{RKHSobolev} is devoted to establishing the formula for the kernel $K_{n}$.
Moreover, the estimate of the norm of $K_{n,w}$ in
$H_2^{(n)}$ is obtained here for every $w\in\C^+$, up to constants from below and from above independent of $w$. In order to explore composition operators on $H_2^{(n)}$, Section \ref{selfmaps} is devoted to preliminary remarks on and examples of
analytic selfmaps of $\C^+$. Finally, the boundedness of composition operators $C_\varphi$ on $H_2^{(n)}$ is discussed in Section \ref{compositionOp}.
This section ends with some examples of rational functions or quotients  of linear combinations of powers of $z$ which are symbols inducing bounded composition operators on $H_2^{(n)}$, $n \ge 1$. In the last section we present some remarks and open questions.

\section{The Sobolev-Lebesgue space ${\mathcal T}^{(n)}_2$}\label{SobolevLebesgue}

Let $W^{-1}$ be the integration operator on the test space $C_c^\infty[0,\infty)$ given by
$$
W^{-1}\varphi(t):=\int_t^\infty \varphi(s)ds,
\quad \varphi\in C_c^\infty[0,\infty), \ t\ge 0.
$$
As a matter of fact, $W^{-1}\varphi\in C_c^\infty[0,\infty)$ for every $\varphi\in C_c^\infty[0,\infty)$ so that we can define the
$n$--th iterate operator $W^{-n}:=W^{-1}W^{-(n-1)}$ of $W^{-1}$ for every $n\in\N$. By straightforward $n$-times application of Fubini's theorem we have
$$
W^{-n}\varphi(t)={1\over (n-1)!}\int_t^\infty (s-t)^{n-1}\varphi(s)ds, \quad \varphi\in C_c^\infty[0,\infty), \ t\ge 0.
$$

The integration operator $W^{-n}$ can be extended in order to act on a number of spaces. Here we consider a Hilbertian setting.

For $n\in\N$, let $L^2(t^n)$ be the Hilbert space of the (classes of) complex measurable functions
$\varphi$ on $(0,\infty)$ such that
$$
\Vert\varphi\Vert_{L^2(t^n)}
:=\left(\int_0^\infty\vert\varphi(t)t^n\vert^2\, dt\right)^{1/2}<\infty.
$$

\begin{proposition}\label{integralderi} Take $n\in \N$ and $\varphi\in L^2(t^n)$. Then,
for  $1\le j\le n-1$, the mapping
$$
(W^{-j}\varphi)(t)
:=\int_t^\infty\int_{t_{j-1}}^\infty\dots
\int_{t_1}^\infty\varphi(s)\, ds\, dt_1\, \dots\, dt_{j-1},\quad t>0,
$$
is well defined and $W^{-j}\varphi\in L^1(a,\infty)$ for every $a>0$.

Moreover, for $1\le k\le n$, we have
$$
(W^{-k}\varphi)(t)={1\over(k-1)!}\int_t^\infty(s-t)^{k-1}\varphi(s)\, ds, \quad t>0,
$$
$W^{-k}\varphi$ is $(k-1)$-times differentiable with $(W^{-k}\varphi)^{(l)}(t)=(-1)^lW^{-(k-l)}\varphi$, that is,
$$
(W^{-k}\varphi)^{(l)}(t)=
{(-1)^{l}\over(k-l-1)!}\int_t^\infty(s-t)^{k-l-1}\varphi(s)\, ds, \quad t>0,
$$
for every $1\le l\le k-1$, and $(W^{-k}\varphi)^{(k-1)}$ is absolutely continuous such that
$(-1)^k(W^{-k}\varphi)^{(k)}=\varphi$.
\end{proposition}

\begin{proof}
For $t>0$, $\varphi\in L^2(t^n)$ and $1\le j\le n-1$, applying Fubini-Tonelli's rule one has
\begin{eqnarray*}
\displaystyle\int_t^\infty\int_{t_{j-1}}^\infty&\dots&
\int_{t_1}^\infty\vert\varphi(s)\vert\, ds\, dt_1\, \dots\, dt_{j-1}
=
{1\over(j-1)!}\int_t^\infty(s-t)^{j-1}\vert\varphi(s)\vert\, ds\cr
&\le&{\Vert\varphi\Vert_{L^2(s^n)}\over(j-1)!}\,
\left(\int_t^\infty(s-t)^{2j-2}s^{-2n}\, dt\right)^{1/2}=C_k t^{-n+j-(1/2)}<\infty,
\end{eqnarray*}
where $C_k$ is a constant involving the Beta function.

The above estimate implies every statement in the text of the current proposition.
\end{proof}

Note that the above formulas also appear in studying the Volterra operator. Recall Hardy's inequality
\begin{equation}\label{hardy}
\int_0^\infty\left(W^{-m}\varphi(t)\right)^2dt\le \left({\Gamma(\frac{1}{2})\over \Gamma(m+{1\over 2})}\right)^2\int_0^\infty (t^m \varphi(t))^2dt,
\end{equation}
where $m\in \N$ and $\varphi:(0,\infty)\to (0, \infty)$ is a Borel function (\cite[ p. 245]{Hardy}).

Using this inequaliy (with $m=n$) one gets $W^{-n}\varphi\in L^2(\R^+)$ for all
$\varphi\in L^2(t^n)$, where $L^2(\R^+)=L^2(t^0)$.
Also, the linear operator $W^{-n}\colon L^2(t^n)\to L^2(\R^+)$ is injective since
$(-1)^n(W^{-n}\varphi)^{(n)}=\varphi$ for all $\varphi\in L^2(t^n)$ in accordance with the last part of the above proposition. These facts enable us to define the Hardy-Lebesgue space $\Tn2$ as follows.

\begin{definition}\label{defi} {\rm Given $n\in \N$, let $\Tn2$ denote the subspace of $L^2(\R^+)$ obtained as the  image
of $L^2(t^n)$ under $W^{-n}$ ,  i.e.,
$$
\Tn2= W^{-n}(L^2(t^n)).
$$

Note that $f^{(n)}, g^{(n)}\in L^2(t^n)$ and we endow $\Tn2$ with the inner product, transferred from $L^2(t^n)$,
$$
\langle f, g\rangle_{(n)}:=\int_0^\infty f^{(n)}(t)\overline{g^{(n)}(t)}t^{2n}dt, \quad  f, g\in \Tn2,
$$
and corresponding norm
$$
\Vert f\Vert_{2, (n)}:=\Vert f^{(n)}\Vert_{L^2(t^n)}=\Vert t^nf^{(n)}\Vert_2, \quad f\in \Tn2,
$$}
so that  the space $(\Tn2, \Vert \quad\Vert_{2, (n)})$ is a Hilbert space.
(In the case $n=0$, we write $\langle \, , \,\rangle_{(0)}=\langle \, , \,\rangle$ for the usual inner product in $L^2(\R^+)$.)
\end{definition}

\medskip
By applying again Hardy's inequality (\ref{hardy}) to
$\varphi(t)=\vert f^{(n)}(t)\vert t^k$ and $m=n-k$, for $n\ge k\ge 0$, we get
\begin{eqnarray*}
\Vert f\Vert_{2, (k)}^2&\le& \int_0^\infty\left( W^{-(n-k)}\varphi(t)\right)^2dt\cr
&\le& {\pi\over \Gamma(n-k+{1\over 2})^2}\int_0^\infty (t^{n} \vert f^{(n)}(t)\vert)^2dt
={\pi\over \Gamma(n-k+{1\over 2})^2}\Vert f\Vert_{2, (n)}^2,
\end{eqnarray*}
for $f\in {\mathcal T}^{(n)}_2$. Thus we have the continuous inclusions
\begin{equation}\label{embedding}
{\mathcal T}^{(n)}_2\hookrightarrow{\mathcal T}^{(k)}_2 \hookrightarrow L^2(\R^+),\quad n\ge k\ge 0.
\end{equation}

In particular we notice that $\Tn2$ consists of all functions
$f:(0, \infty)\to \C$ such that the derivatives $f',\cdots, f^{(n-1)}$ on $(0,\infty)$ exist, $f^{(n-1)}$
is absolutely continuous on $(0,\infty)$, and the maps $t\mapsto t^kf^{(k)}(t)$ belong to $ L^2(\R^+)$ for all $k=0,1,\dots,n$.
Moreover, taking
$\varphi=(-1)^n f^{(n)}$ in the statement and proof of Proposition \ref{integralderi} one obtains
\begin{equation}\label{pointestim}
\vert f^{(k)}(t)\vert\leq C_{n,k}t^{-k-(1/2)}\Vert f\Vert_{2,(n)}, \quad 0\le k\le n-1,\ t>0.
\end{equation}

\begin{remark}
\normalfont
The spaces $\Tn2$, $n\in\N$, have been introduced in \cite{Roy}. The presentation of the properties of such spaces done before is slightly different from the one of \cite{Roy}.
\end{remark}

\medskip
Now, our aim is to present an equivalent expression of the inner product of the space $\Tn2$.

\begin{lemma}\label{clave} Given $n \ge 0$, we define the $(n+1)$-square matrix $C_n=(c_{i,j})_{0\le i, j\le n}$ by
$$
c_{i,j}= \left\{\begin{array}{lr}
0, &  \text{for } i<j;\\
\displaystyle{\binom{i}{j}{i!\over j!},}&  \text{for } i\ge j.
\end{array}\right.
$$
\begin{itemize}
\item[(i)] Then the matrix $C_n$ is invertible and $C^{-1}_n= ((-1)^{i+j}c_{i,j})_{0\le i, j\le n}$.
\item[(ii)] For $f\in \Tn2$ and $t>0$, we have that
\begin{eqnarray*}
(t^{n}f)^{(n)}(t) &=& \sum_{k=0}^{n}c_{n,k}t^{k}f^{(k)}(t), \\
t^{n}f^{(n)}(t) &=& \sum_{k=0}^{n}(-1)^{k+n}c_{n,k}(t^{k}f)^{(k)}(t). \\
\end{eqnarray*}
\end{itemize}
\end{lemma}

\begin{proof}
Part (i) of the lemma is a technical exercise and it is left to the reader. The first equality in part (ii) is the $n$-th derivation rule applied to the product $t^n f$, and then the second equality in (ii) is a consequence of part (i).
\end{proof}

As an interesting fact we remark that $c_n=\sum_{k=0}^n c_{n,k}$, $n\ge 0$, is a noted sequence (see reference A002720 in  ``The On-Line Encyclopedia of Integer Sequences'' by N.J.A. Sloane).

\medskip
In the next lemma, $e_\lambda(t):=e^{-\lambda t}$ for $\lambda\in\C^+$, $t\ge 0$.

\begin{lemma}\label{dense} Let $n$ be a nonnegative integer.
\begin{itemize}
\item[(i)] For every $\lambda\in\C^+$, the function $e_{\lambda}$ belongs to $\Tn2$ and
$$
\Vert e_\lambda \Vert_{2,(n)}
= \left({(2n)!\over 2^{2n+1}}\right)^{1\over 2}{\vert \lambda\vert^n\over (\Re \lambda)^{n+{1\over 2}}}.
$$
\item[(ii)]  The linear subspace $\mathcal E$ spanned by  $\{e_{\lambda}:\Re\lambda>0\}$ is dense in $\Tn2$.
\end{itemize}

\end{lemma}
\begin{proof} The proof of part (i) is straightforward. Now take $f \in \Tn2$ such that
$\left\langle f,e_\lambda \right\rangle_{(n)}=0$ for $\lambda \in \CC^+$, i.e.,
$$
0=(-1)^n\lambda^n\int_0^\infty t^{2n}e^{-\lambda t}f^{(n)}(t) dt= \lambda^n{\mathcal L}(t^nf^{(n)})^{(n)}(\lambda), \quad \lambda \in \C^+,
$$
whence one has that ${\mathcal L}(t^nf^{(n)})(\lambda)$ must be a polynomial. Since $t\mapsto t^nf^{(n)}$ is in $L^2(\R^+)$ that polynomial is zero by the Paley-Wiener theorem. Then, as
the Laplace transform ${\mathcal L} $ is one-to-one, we conclude that $t^nf^{(n)}(t)=0$ in $L^2(\RR^+)$, whence $f=0$ in $\Tn2$ and therefore $\mathcal E$ is dense in $\Tn2$.
\end{proof}

The following result gives us the desired expression of  the inner product of the space $\Tn2$. In the proof we use Laguerre polynomials and Legendre polynomials.

\begin{proposition} \label{teoLii} Let $f,g \in \Tn2$. Then
\begin{equation} \label{eqinnerma}
\left\langle f, g \right\rangle_{(n)} = \left\langle (t^{n}f)^{(n)},(t^{n}g)^{(n)} \right\rangle.
\end{equation}
\end{proposition}

\begin{proof} Let $e_{\lambda}(t):=e^{-\lambda t}$ and $e_{\mu}(t):=e^{-\mu t}$, with $\lambda,\mu > 0$. Then, by definition \ref{defi},
$$
\left\langle e_{\lambda},e_{\mu} \right\rangle_{(n)}=\left\langle t^{n}e_{\lambda}^{(n)},t^{n}e_{\mu}^{(n)} \right\rangle
= (\lambda\mu)^{n} \int_{0}^{\infty} t^{2n} e^{-(\lambda+\mu)t} dt = \frac{(\lambda\mu)^{n} (2n)!}{(\lambda+\mu)^{2n+1}}.
$$
For the second inner product in (\ref{eqinnerma}), note that
$$
(t^{n}e_{\lambda})^{(n)}(t) = n! \,\La_{n}(\lambda t),
$$
where $\La_{n}$ is the Laguerre polynomial of degree $n$. Therefore,
$$
\left\langle (t^{n}e_{\lambda})^{(n)},(t^{n}e_{\mu})^{(n)} \right\rangle = (n!)^{2} \int_{0}^{\infty}\La_{n}(\lambda t)\La_{n}(\mu t) e^{-(\lambda+\mu)t}dt.
$$
Then, according to \cite[7.414 (2)]{gra-ryz},
$$
\int_{0}^{\infty}\La_{n}(\lambda x)\La_{n}(\mu x)e^{-bx}dx = \frac{(b-\lambda-\mu)^{n}}{b^{n+1}} \Po_{n}\left(\frac{b^{2}-(\lambda+\mu)b+2\lambda\mu}{b(b-\lambda-\mu)}\right),
$$
for $\Re b >0$, where $\Po_{n}$ is the Legendre polynomial of degree $n$. Therefore,
\begin{eqnarray*}
&\quad&\langle (t^{n}e_{\lambda})^{(n)},(t^{n}e_{\mu})^{(n)} \rangle = (n!)^{2} \displaystyle\lim_{b\rightarrow \lambda+\mu} \frac{(b-\lambda-\mu)^{n}}{b^{n+1}} \Po_{n}\left(\frac{b^{2}-(\lambda+\mu)b+2\lambda\mu}{b(b-\lambda-\mu)}\right)\cr
&\quad&\qquad  = (n!)^{2} \displaystyle\lim_{b\rightarrow \lambda+\mu} \frac{(b-\lambda-\mu)^{n}}{b^{n+1}} \frac{(2n)!}{2^{n}(n!)^{2}} \left(\frac{b^{2}-(\lambda+\mu)b+2\lambda\mu}{b(b-\lambda-\mu)}\right)^{n}\cr&\quad&\qquad  = \displaystyle\frac{(2n)!(\lambda\mu)^{n}}{(\lambda+\mu)^{2n+1}},
\end{eqnarray*}
because $(2n)!\ 2^{-n}(n!)^{-2}$ is the leading coefficient of $\Po_{n}$, as can be seen in \cite[Section 5.4.2]{mag-obe-son}.  We conclude that
$$
\left\langle e_{\lambda},e_{\mu} \right\rangle_{(n)} = \left\langle (t^{n}e_{\lambda})^{(n)},(t^{n}e_{\mu})^{(n)} \right\rangle, \quad \lambda, \mu \in \CC^+,
$$
and then by linearity
$\left\langle f,g \right\rangle_{(n)} = \left\langle (t^n f)^{(n)},(t^{n} g)^{(n)} \right\rangle,$
for $f,g\in\mathcal E$.
Finally for arbitrary $f, g \in\Tn2$ it suffices to apply
Lemma \ref{clave} (ii), (\ref{embedding}) and the density of
$\mathcal E$ in $\Tn2$ (Lemma \ref{dense} (ii)) to obtain
$\left\langle f, g \right\rangle_{(n)} = \left\langle (t^{n}f)^{(n)},(t^{n}g)^{(n)} \right\rangle$. The proof is over.
\end{proof}

\medskip
To finish this section we show an example of function in the space
$\Tn2$ which will be used, in Section 5 below, as a key to introduce reproducing kernels in Hardy-Sobolev spaces on the half-plane $\C^+$.

\begin{proposition} \label{key} For $w \in \C^+$, $n\in \N$ and $t>0$, define
\begin{eqnarray*}
g_{w, n}(t)
&:=& \int_{t}^{\infty}\frac{(s-t)^{n-1}}{s^{n}}
\int_{0}^{1}\frac{(1-x)^{n-1}}{((n-1)!)^{2}}e^{-sxw} dx ds\cr
\end{eqnarray*}

Then $\gwn$ belongs to $\Tn2$ with $n$-th derivative satisfying
$$
t^{n}\gwn^{(n)}(t) = \frac{(-1)^{n}}{t^{n}}\int_{0}^{t}\frac{(t-s)^{n-1}}{(n-1)!}e^{-w s} ds.
$$
In particular, $\displaystyle\Vert g_{w, n}\Vert_{2, (n)}
\le {2\log2\over\sqrt{\Re w}}$.
\end{proposition}

\begin{proof} Take $w\in\C^+$ and $n\in\N$.
Put
$$
\varphi_{w,n}(t):=t^{-n}\int_{0}^{1}\frac{(1-x)^{n-1}}{(n-1)!}\ e^{-txw} dx, \quad t>0.
$$

Then $\varphi_{w,n}\in L^2(t^n)$. In fact,
\begin{eqnarray*}
\left(\int_0^\infty\vert\varphi(t)\vert^2 t^{2n}\ dt\right)^{1/2}
&\le&\left(\int_0^\infty
\left(\int_0^1{(1-x)^{n-1}\over(n-1)!}\ e^{-(\Re w)tx} dx\right)^2 dt\right)^{1/2}\cr
&\le&\int_0^1\frac{(1-x)^{n-1}}{(n-1)!}
\left(\int_0^\infty e^{-2(\Re w)tx}dt\right)^{1/2}dx\cr
&=&{\sqrt{\pi/2}\over{\Gamma(n+(1/2))\sqrt{\Re w}}}
\end{eqnarray*}
where the second inequality is Minkowski`s inequality .

Hence, $g_{w,n}=W^{-n}\varphi_{w,n}$ is in $\Tn2$ and, moreover, by
Proposition \ref{integralderi},
$$
{\gwn^{(n)}(t)} ={(-1)^{n}\over t^n}\int_{0}^{1}\frac{(1-x)^{n-1}}{(n-1)!}e^{-txw} dx
=\frac{(-1)^{n}}{t^{2n}}\int_{0}^{t}\frac{(t-s)^{n-1}}{(n-1)!}e^{-w s} ds,
$$
for every $t>0$. Therefore,
\begin{eqnarray*}
\Vert \gwn\Vert_{2,(n)}^2&=&\|t^{n}\gwn^{(n)}\|_2
= \displaystyle\int_{0}^{\infty} \Big|\int_{0}^{1}\frac{(1-x)^{n-1}}{(n-1)!}e^{-txw} dx\Big|^{2}dt\cr
&\leq& \displaystyle\frac{1}{(n-1)!} \int_{0}^{\infty} \left(\int_{0}^{1}e^{-\Re w tx} dx\right)^{2}dt\cr
&=& \displaystyle\frac{1}{(n-1)!(\Re w)^2} \int_{0}^{\infty}\frac{(1-e^{-\Re w t})^{2}}{t^{2}}dt
= \displaystyle\frac{2\log2}{(n-1)!\Re w}.
\end{eqnarray*}
\end{proof}

\section{{The Hardy-Sobolev spaces $H^{(n)}_2$}}\label{HardySobolev}

Recall that $H_2$ is the Hilbertian Hardy space over the right-hand half-plane $\C^+$ endowed with the norm defined in (\ref{Norm01}), which corresponds to the inner product
$$
(F\mid G)=\frac{1}{2\pi }\int_{-\infty}^\infty F^*(it)\overline{G^*(it)}dt, \quad F,G\in H_2,
$$
where $F^*(t)=\lim_{z\to it}F(s)$ a.e. $t\in\R$ for every $F,G\in H_2$; see \cite[Chapter VI]{Duren} and \cite[Chapter 8]{Hoffman}.
The classical Paley-Wiener theorem states that the Laplace transform ${\mathcal L}: L^2(\R^+) \to H_2$,
$$
{\mathcal L}(f)(z)=\int_0^\infty f(t)e^{-zt}dt, \quad f\in L^2(\R^+), \quad z\in \C^+,
$$
is  an isometric isomorphism; i.e., $F, G\in H_2 $ if and only if there exist unique $ f,g\in L^2(\R^+)$ such that
$ F=\LL f$ and  $ G=\LL g$ and
$$
\left\langle f,g \right\rangle=\left( \LL f\mid \LL g \right) ,  \quad f,g\in L^2(\R^+),
 $$
(see for example \cite[Theorem 19.2]{Rud87}).
Recall that the space $\Tn2$ introduced in the preceding section is a subspace of $L_2(\R^+)$. The scale of  range spaces $\LL(\Tn2)$, $n\in\N$, is characterized in \cite{Roy}. In the following we provide a shorter proof of that characterization.

\medskip
Let $k$ a positive integer. Integrating by parts $k$ times, one gets
\begin{equation}\label{igua2}
z^{k}\LL h(z) = \LL(h^{(k)})(z) + \sum_{j=0}^{k-1}z^{k-1-j}h^{(j)}(0), \quad z\in \CC^+,
\end{equation}
provided the function $h$ is differentiable $k-1$ times on the interval $(0,\infty)$, the derivatives $h^{(j)}$ are continuous
on $[0,\infty)$, and the functions
$h^{(j)}e^{-z(\cdot)}$ belong to $L^1(\R^+)$ for all $z\in\C^+$ and $j=0,1,\dots,k$.

\begin{lemma}\label{derilaplace} For every $n\in\N$, $z\in\overline\C^+\setminus\{0\}$ and
$f\in\mathcal T^{(n)}_2$,
\begin{eqnarray*}
(-1)^k z^k(\mathcal{L}f)^{(k)}(z)&=&\sum_{j=0}^{k}{k\choose j}{k!\over j!}\mathcal{L}(t^{j}f^{(j)})(z)\ ,
\qquad k=0,1,\dots,n;\cr
(-1)^k \mathcal{L}(t^kf^{(k)})(z)&=&\sum_{j=0}^{k}{k\choose j}{k!\over j!}z^j(\mathcal{L}f)^{(j)}(z)\ ,
\qquad k=0,1,\dots,n;\cr
\end{eqnarray*}
\end{lemma}

\begin{proof}
Let $j, k$ be integers such that $0\le j\le k\le n$. Using Leibniz's rule of derivation in $(t^kf)^{(j)}$ and the bound given by
(\ref{pointestim}) one finds that
$$
(t^kf)^{(j)}(0):=\lim_{t\to0^+}(t^kf)^{(j)}(t)=0, \quad j\in\{0,1,\dots,k-1\}.
$$
Then taking $h=t^kf$ in (\ref{igua2}) we get
\begin{equation}\label{laplacederi}
\LL((t^kf)^{(k)})(z) = z^{k}\LL(t^kf)(z)=(-1)^kz^k(\LL f)^{(k)}(z), \quad z\in\C^+.
\end{equation}

Now, using Lemma \ref{clave}, part (ii), we obtain
$$
(-1)^kz^k(\LL f)^{(k)}(z)=\sum_{j=o}^k c_{k,j}\LL(t^jf^{(j)})(z), \quad z\in\C^+,
$$
where coefficients $c_{k,j}$ are as in that lemma. Applying the invertibility of the matrix $(c_{k,j})$ one also derives
$$
\LL(t^jf^{(j)})(z)=(-1)^j\sum_{k=o}^j c_{j,k}z^k(\LL f)^{(k)}(z), \quad z\in\C^+,
$$
and the proof is over.
\end{proof}

\begin{definition}
Given $n\geq 0$ a fixed integer, let $H^{(n)}_2$ denote the linear space consisting of all analytic functions $F$ on $\C^+$ such that
\begin{equation}
z^kF^{(k)}\in H_2, \quad k=0,1,\dots,n.
\end{equation}
We use the notation $H^{(0)}_2=H_2$.
\end{definition}

Here is the Paley-Wiener theorem for Hardy-Lebesgue spaces.

\begin{theorem} \label{teoroyo} {\rm (Extended Paley-Wiener Theorem)} Let $\LL$ be the Laplace transform $\LL:L_2(\R^+)\to H_2$. Then, for $n\in\NN$,
$$
\LL(\TniiRma)=\HiinCma.
$$

Consequently, $\HiinCma$ is a Hilbert space with respect to the inner product
$$
\left(F\mid G\right)_{(n)}:={1\over 2\pi}\int_{-\infty}^{\infty}
t^{2n}(F^*)^{(n)}(it)\overline{(G^*)^{(n)}(it)}\ dt, \quad F,G\in\HiinCma,
$$
and corresponding norm
$$
\Vert F\Vert_{2,(n)}:=\Vert z^nF^{(n)}\Vert_{2}, \quad F\in\HiinCma,
$$
and the Laplace transform is an isometric isomorphism from $\TniiRma$ onto $\HiinCma$;
that is, it is an onto linear isomorphism
having the property
$$
\left\langle f,g \right\rangle_{(n)}=\left( \LL f\mid\LL g \right)_{(n)},
\quad f,g\in\TniiRma.
$$
\end{theorem}

\begin{proof} Given $f\in\TniiRma$, then  $t^k f^{(k)}\in L^2(\RR^+)$ and therefore
$z^k(\mathcal{L}f)^{(k)}\in {H}_{2}$ for $0\le k\le n$ by Lemma \ref{derilaplace}. Thus $\mathcal{L}f\in \HiinCma$.

Conversely, take  $F\in\HiinCma\subset H_2$. By the classical Paley-Wiener theorem,  there exists
$f\in L^2(\RR^+)$  such that $F=\LL f$.
We want to prove that, in fact, $f\in\TniiRma$. To this aim denote by $G$  the analytic function defined by
$$
G(z):= (-1)^n\sum_{j=0}^{n}{n\choose j}{n!\over j!}z^jF^{(j)}(z),\qquad z\in \CC^+.
$$
Since $F\in\HiinCma$, we have that $G$ belongs to $H_{2}$  and so there also exists $g\in L^2(\RR^+)$ such that
$G=\LL g$.
As in Definition \ref{defi},  the function $\tilde f$ belongs to $\TniiRma$, where
$$
\tilde f(s):={1\over (n-1)!}\int_s^\infty {(s-t)^{n-1}\over t^n}g(t)dt, \qquad s>0.
$$
Note that $ g= t^n (\tilde{f})^{(n)}$ by Theorem \ref{integralderi}. Now we apply Lemma \ref{derilaplace} to get, for
$z\in \CC^+$,
$$
G(z)= \LL g(z)=\LL( t^n (\tilde{f})^{(n)})(z)=  \sum_{j=0}^{n}(-1)^n{n\choose j}{n!\over j!}z^j\LL(\tilde{f})^{(j)}(z).
$$
Since $F=\LL f$, it follows from Lemma  \ref{derilaplace} that
$$
\left(z^n \LL(\tilde{f}-f)\right)^{(n)}(z)=(-1)^n\sum_{j=0}^{n}{n\choose j}{n!\over j!}z^j\LL(\tilde{f}-f)^{(j)}(z)=0,
\quad z\in \CC^+,
$$
so that $z^n \LL(\tilde{f}-f)(z)= P_n(z)$ where $P_n$ is a polynomial of degree less than or equal to $n-1$. Since
$P_n(z)z^{-n}$ lies in $H_2$ it must be identically null on $\C^+$; that is, $\LL(\tilde{f}-f)=0$. Hence
$f=\tilde{f}\in\TniiRma$ by the injectivity of $\LL$. All in all, we have shown that $\LL(\TniiRma)=\HiinCma$.

In order to finish the proof of the theorem we must only show the equality between  the inner products in
$\TniiRma$  and $\HiinCma$, respectively. This is as follows.

Take $f,g\in\TniiRma$ with $F=\LL f, G=\LL g$ in $\HiinCma$.
By Theorem \ref{teoLii} and the isometry provided by the classical Paley-Wiener Theorem,
\begin{eqnarray*}
\left\langle f,g \right\rangle_{(n)}& :=& \left\langle t^{n}f^{(n)},t^{n}g^{(n)} \right\rangle
 = \left\langle (t^{n}f)^{(n)},(t^{n}g)^{(n)} \right\rangle \\
 &=& \left( \LL((t^{n}f)^{(n)})\mid \LL((t^{n}g)^{(n)}) \right)\\
&=& \left( (-1)^nz^{n}(\LL f)^{(n)}(z)\mid (-1)^nz^{n}(\LL g)^{(n)}(z) \right)\\
&=& \left( F\mid G \right)_{(n)},
\end{eqnarray*}
where we have used (\ref{laplacederi}) in the last but one equality.
\end{proof}

\begin{corollary}
For every $m\ge n\ge0$, the linear inclusion
$$
H^{(m)}_2\hookrightarrow H^{(n)}_2
$$
is continuous with dense range.
\end{corollary}

\begin{proof} The continuity of the inclusion mapping $H^{(m)}_2\hookrightarrow H^{(n)}_2$ follows from Theorem \ref{teoroyo}
and (\ref{embedding}). As for the density, let $\mathcal K$ be the linear space spanned by all functions of type $K_\lambda$,
$\Re\lambda>0$. Then $\mathcal L(\mathcal E)=\mathcal K$ where $\mathcal E$ is the space defined in Lemma \ref{dense} (ii). Clearly,
$\mathcal K\hookrightarrow H^{(m)}_2\hookrightarrow H^{(n)}_2$ for $m\ge n\ge0$, whence the desired density follows from
Lemma \ref{dense} (ii) and Theorem \ref{teoroyo} again.
\end{proof}

\section{The Hardy-Sobolev spaces on the disc}\label{HardySobolev2}
Recall that the Hilbertian Hardy space $H_2(\D)$ over the unit disc $\mathbb{D}$ is endowed with the norm
$$
\| F\| _{2,\D}:=\sup_{0<r<1} \left(\frac{1}{2\pi }\int_{\theta\in [0, 2\pi]} |F(re^{i\theta})|^2\, d\theta\right)^{1\over 2}.
$$

It is well known that one can use  Cayley's transform to navigate between spaces $H_2(\D)$ and $H_2$, and subspaces of them: The Cayley transform is defined by
$\gamma (\lambda)=(1+\lambda)/(1-\lambda)$, for all $\lambda\in \D$.
Then,
$F\in H_2$ if and only if $F_{\D}\in (1-\lambda)H_2(\D)$, where
\begin{equation}\label{FunctionG02}
F_\D(\lambda):=F\left(\frac{1+\lambda}{1-\lambda}\right), \quad \lambda\in \D;
\end{equation}
see \cite[Theorem p. 129]{Hoffman}. See also \cite[Theorem 13.17]{JM} for the relations between $H^p(\D)$  and $H^p(\C^+)$ for $0< p\le \infty$.
Functions in ${H^{(1)}_2}$ can be transferred to a subspace contained in
$H_2(\D)$ as follows.

\begin{proposition}\label{PropositionUnitDisc}
Let $F$ be an analytic function on $\C^+$ and $F_\D$ the function on the unit disc $\mathbb{D}$ given by {\rm (\ref{FunctionG02})}.
Then $F\in {H^{(1)}_2} $ if and only if
\begin{equation}\label{FunctionG01}
F_\D\in (1-\lambda)H_2(\D)\quad \mbox{and}\quad F_\D^\prime \in (1+\lambda)^{-1}H_2(\D).
\end{equation}
If so, then the following equality holds:
\begin{equation}\label{NormEquality}
\sqrt2\| F\|_{2,(1)} =\| (1+\lambda)F^\prime_\D (\lambda)\|_{2,\D}.
\end{equation}
\end{proposition}

\begin{proof}
The condition  $zF'\in H_2$ is equivalent to
\begin{equation}\label{Belongs01}
\frac{1+\lambda}{1-\lambda}F^\prime \left(\frac{1+\lambda}{1-\lambda}\right)
\in (1-\lambda)H_2(\D).
\end{equation}
Since, obviously,
$$
(F_\D)^\prime (\lambda)=F^\prime
\left(\frac{1+\lambda}{1-\lambda}\right)\frac{2}{(1-\lambda)^2},
$$
it follows that relation (\ref{Belongs01}) is equivalent to $F^\prime_\D \in (1+\lambda)^{-1}H_2(\D)$.
Thus $F\in {H^{(1)} _2}$ if and only if (\ref{FunctionG01}) holds. Also, the norm equality (\ref{NormEquality}) is valid,
by \cite[p. 130]{Hoffman}.
\end{proof}

Of course Proposition \ref{PropositionUnitDisc} can be extended for $n>1$, i.e., to identify the subspaces in
$(1-\lambda)H_2(\D)$  formed by functions $F_\D$ with $F\in H_2^{(n)}$ for $n> 1$. For example,
for $n=2$ and $F\in H_2^{(2)}$, we have that
$$
(F_\D)''(\lambda)={-4\over (1-\lambda)^3}F'\left(\frac{1+\lambda}{1-w}\right)
+{4\over (1-\lambda)^4}F''\left(\frac{1+\lambda}{1-w}\right),\quad \lambda\in \D,
$$
and we conclude that $(1+\lambda)^2(1-\lambda)F''_\D\in H_2(\D)$.

In this way one could deal with Hardy-Sobolev type spaces on the unit disc instead on the right-hand half-plane.
However we stay working in the half-plane setting, where one can make use of the Laplace transform.

Nonetheless, let us show an interesting property of $H_2^{(n)}$ by arguing on its copy in $H_2(\D)$

\begin{corollary}
If $F\in H^{(n)}_2$, then the functions $z^{k-1}F^{(k-1)}$ are extensible by continuity at all nonzero points
on the imaginary axis, the point at infinity included, for $k=1,2, \dots ,n$.
\end{corollary}

\begin{proof}
Take $F\in H^{(1)}_2$. By condition (\ref{FunctionG01}),
$F_\D\in (1-\lambda)H_2(\D)\subseteq H_2(\D)$, $(1+\lambda)F^\prime_\D \in H_2(\D)$
and clearly $[(1+\lambda)F_\D(\lambda)]^\prime=(1+\lambda)F^\prime_\D (\lambda)+F_\D(\lambda)$.
It follows that
$[(1+\lambda)F_\D(\lambda)]^\prime \in H_2(\D)\subseteq H_1(\D)$. Then, by \cite[Theorem 3.11]{Duren},
the function $(1+\lambda)F_\D(\lambda)$ is extensible by continuity on the unit circle and so, $F_\D$ extends by continuity to that circle except possibly at $-1$. Since the Cayley transform $\gamma$
maps $-1$ to $0$ and $1$ to $\infty $, the conclusion on the extensibility of $F$ by continuity follows.

Now take $F\in H^{(n)}_2$ for $n>1$. Then the identity
$$
z\frac{d}{dz}\left( z^{k-1}F^{(k-1)}(z)\right)=(k-1)z^{k-1}F^{(k-1)}(z)+z^{k}F^{(k)}(z), \quad z\in \C^+,
$$
implies that $z^{k-1}F^{(k-1)}(z)$ belongs to $H^{(1)}_2$ for $1\le k\le n$. Thus the continuity on the imaginary axis of the functions
$z^kF^{(k)}$, $k=1, \dots, n-1$, follows by that and the extensibility by continuity of $H^{(1)}_2$--functions.
\end{proof}

\section{Reproducing kernel of $H^{(n)}_2$}\label{RKHSobolev}

Given a set $\Omega$ one says that ${\mathcal H}$ is a (complex) reproducing kernel Hilbert space (RKHS, for short) on $\Omega$ if it is a Hilbert space of complex functions on $\Omega$ with inner product $\langle\cdot\mid\cdot\rangle$ such that the evaluation functional
$$
\delta_s\colon{\mathcal H}\to\C, \ f\mapsto f(s),
$$
is bounded for all $s\in \Omega$. Thus for every $s\in \Omega$ there exists a unique
$k_s\in{\mathcal H}$ such that $f(s)=\langle f\mid k_s\rangle$ for every $f\in{\mathcal H}$. Functions $k_s$ are called kernel-functions, and the span generated by of all kernel-functions is dense in $\mathcal H$. One defines
$K(s,t)=\langle k_t\mid k_s\rangle $, $s,t\in \Omega$, which is said to be the reproducing kernel of $\mathcal H$. This map $K$ is a
nonnegative, or positive semidefinite, kernel (called sometimes positive semidefinite Moore matrix) in the sense of Aronszajn--Moore, meaning that it satisfies
$$
\sum_{i,j=1}^nK(s_i,s_j)c_i\overline{c_j}\geq 0,\quad n\geq 1,\,\, c_i\in \mathbb{C},\,\, s_i\in \Omega,\,\, i,j=1,2,\dots ,n.
$$
A basic reference for nonnegative kernels is \cite{Aronszajn}.

Point evaluation functionals on $H_2$ are continuous, so that $H_2$ is a RKHS. In fact,
$F(w)=(F\mid K_w)$ with $K_w:=(z+\overline w)^{-1}$ for all $z\in\C^+$ and $w\in\C^+$, as mentioned in
Section \ref{intro}. Let $n$ be a natural number. Since $H_2^{(n)}$ is continuously contained in $H_2$ it follows that $H_2^{(n)}$ has all point evaluations continuous, so it is also a RKHS.
In this section we determine the nonnegative kernel of $H_2^{(n)}$ and provide estimates of the norm of kernel functions.

\begin{theorem}\label{RKHS}
Let $n$ be a natural number. Then the function $K_{n}$ defined on $\C^+\times\C^+$ by
$$
K_{n}(z,w)=\frac{1}{((n-1)!)^2}\int _0^1\int _0^1\frac{(1-t)^{n-1}(1-s)^{n-1}}{zt+s\overline{w}}\, ds\, dt,
\ z,w \in \C ^+,
$$
is the reproducing kernel of  $H_2^{(n)}$; that is, if  $K_{n,w}(z):=K_{n}(z,w)$ then
$$
K_{n,w}\in H_2^{(n)} \ \hbox { and } F(w)=(F\mid K_{n,w}) \ \hbox { for all } F\in H_2^{(n)}.
$$
\end{theorem}

\begin{proof} Let $w \in \C^+$ and let $g_{{w},n}$ be the function discussed in Proposition \ref{key}.
Then $g_{{w},n}\in\TniiRma$ and moreover, using the change of  variables $s=ty^{-1}$ and taking the Laplace transform
we have for every $z\in\C^+$,
\begin{eqnarray*}
\LL (g_{{w},n})(z)&=& \int_{0}^{1} \int_{0}^{1}
\frac{(1-y)^{n-1}}{(n-1)!} \frac{(1-x)^{n-1}}{(n-1)!}
{\mathcal L}(e^{-(\cdot)\ w x/y})(z) dx {dy\over y}\\
&=&  \int_{0}^{1} \int_{0}^{1} \frac{(1-y)^{n-1}}{(n-1)!} \frac{(1-x)^{n-1}}{(n-1)!} \frac{1}{xw+zy} dx dy
=K_{n,\overline{w}}(z).
\end{eqnarray*}

On the other hand, using the integral formula given for $g_{w, n}^{(n)}$ in Proposition \ref{key} we have,
for any $f\in\TniiRma$,
\begin{eqnarray*}
{\mathcal L}(f)(\overline{w})&=&\int_{0}^{\infty}f(t)e^{-\overline{w} t}dt
= (-1)^{n}\int_{0}^{\infty} \int_{t}^{\infty}\frac{(t-s)^{n-1}}{(n-1)!}f^{(n)}(s) ds\ e^{-\overline{w} t}dt\cr
&=&
(-1)^{n} \int_{0}^{\infty}f^{(n)}(t)\int_{0}^{t}\frac{(t-s)^{n-1}}{(n-1)!}e^{-\overline{w} s} ds dt\cr
&=& \displaystyle\int_{0}^{\infty}f^{(n)}(t)t^{2n} \overline{\gwn^{(n)}(t)}dt =\langle f, g_{w, n}\rangle_{2,(n)}.
\end{eqnarray*}

Now take $F\in H^{(n)}_2$. By Theorem \ref{teoroyo}, there exists a unique $f\in \TniiRma$
such that ${\mathcal L}(f)=F$ with
$$
(F\mid K_{n,w} )_{(n)}=
\langle f, g_{\overline w,n} \rangle_{2,(n)}=F(w), \quad w \in \C^+,
$$
as we needed to show.
\end{proof}

\begin{remark}
\normalfont
By expanding $(1-t)^{n-1}$ and $(1-s)^{n-1}$ in the integral expression of $K_n$ one gets
\begin{eqnarray*}K_{n,w}(z)
&=&
\left(\displaystyle\sum_{k=0}^{n-1}\frac{1}{(n+k)!(n-1-k)!} \frac{wc^{k}}{z^{k+1}}\right)\log\left(\displaystyle\frac{z+w c}{w c}\right) \cr&\qquad&\qquad\qquad + \left(\displaystyle\sum_{k=0}^{n-1}\frac{1}{(n+k)!(n-1-k)!} \frac{z^{k}}{wc^{k+1}} \right)\log\left(\displaystyle\frac{z+wc}{z}\right)\cr
&\qquad&\qquad\qquad + \displaystyle\sum_{k=1}^{n-1} \sum_{j=0}^{k-1} \frac{(-1)^{k+j}}{(k-j)(n+k)!(n-1-k)!} \frac{z^{j}}{wc^{j+1}}\cr&\qquad&\qquad\qquad +
\displaystyle\sum_{k=1}^{n-1} \sum_{j=0}^{k-1} \frac{(-1)^{k+j}}{(k-j)(n+k)!(n-1-k)!} \frac{wc^{j}}{z^{j+1}}.\end{eqnarray*}

In particular, for $n=1$, $n=2$, and $n=3$ we have respectively
\begin{eqnarray*}
K_{1,w}(z)&=&{1\over z}\log\left({\overline{w}+z\over\overline{w} }\right)+{1\over \overline{w}}\log\left({\overline{w}+z\over z }\right),\cr
K_{2,w}(z)&=&-{1\over 6z}+{3z+\overline{w}\over 6z^2}\log\left({\overline{w}+z\over\overline{w} }\right)
-{1\over 6\overline{w}}+{3\overline{w}+z\over 6\overline{w}^2}\log\left({\overline{w}+z\over z }\right),\cr
K_{3,w}(z)&=&-{(9z+2\overline{w})\over 240z^2}
+{10z^2+5z\overline{w}+\overline{w}^2\over 120z^3}\log\left({\overline{w}+z\over\overline{w}  }\right)\cr
&\quad&-{(9\overline{w}+2z)\over 240\overline{w}^2}+{10\overline{w}^2+5z\overline{w}+z^2\over 120\overline{w}^3}
\log\left({\overline{w}+z\over z }\right).
\end{eqnarray*}
\end{remark}

\medskip
Next, we proceed to estimate the norm of the kernel function. We claim that
$\Vert K_{n,w}\Vert_{2,(n)}\sim \vert w\vert^{-1/2}$ for $w\in\C^+$ (up to lower and upper constants depending on $n$). To show this, we appeal to the trigonometric form of the integral definining $K_n$.

\begin{lemma}\label{integ}
For $\theta\in ({-\pi/2}, {\pi/2})$, put
$$
I(\theta):=\displaystyle{\int_{0}^{1}\frac{\cos\theta}{t^{2}+1+2t\cos(2\theta)}dt}.
$$
Then $\displaystyle{I(0)={1\over2}}$ and
$$
I(\theta)=  \frac{1}{2} \frac{\theta}{\sin\theta}, \hbox{ if } \theta\not=0,
$$
\hbox { whence } $\displaystyle{
{1\over 2}\le I(\theta)\le {\pi \over 4}}$.
\end{lemma}

\begin{proof}
For $\theta=0$, it is straightforward to check that $I(0)=1/2$.

For $0<\theta<{\pi \over 2}$, we have
$I(\theta)=I(-\theta)$ and
\begin{eqnarray*}
I(\theta) &= &{\cos\theta}{}\int_{0}^{1}\frac{dt}{(t+\cos(2\theta))^{2}+\sin^{2}(2\theta)}\cr
&=&\frac{1}{2|\sin\theta|}\int_{0}^{\frac{1}{|\sin(2\theta)|}}\frac{dx}{\left(x+\frac{\cos(2\theta)}{|\sin(2\theta)|}\right)^{2}+1} \cr &=&\frac{1}{2|\sin\theta|}\int_{\frac{\cos(2\theta)}{|\sin(2\theta)|}}^{\frac{\cos\theta}{|\sin\theta|}}\frac{dy}{y^{2}+1} = \frac{1}{2} \frac{\theta}{\sin\theta},
\end{eqnarray*}
where we have  changed the variables $\displaystyle{dy=dx+\cos(2\theta)\vert\sin(2\theta)\vert^{-1}}$.

Finally, since $1\le (\theta/\sin\theta)\le\pi/2$ the desired result follows.
\end{proof}

In the following, we use the relation
$$
\| K_{n,z}\|^2_{2,(n)} = \left( K_{n,z}\mid K_{n,z}  \right)_{(n)} = K_{n,z} (z), \quad n\ge1, z\in\C^+.
$$

\begin{theorem} \label{cota} Let $n\in\NN$,  $z=\vert z\vert e^{i \theta}\in\CCma$ and $\theta \in (-\pi/2, \pi/2)$.
Then
\begin{equation}
\label{inequ}
\frac{1}{\Gamma(n)\sqrt{{ (2n-1)}}}{1\over\sqrt{\vert z\vert}}
\leq \|K_{n,z}\|_{2,(n)} \leq \frac{\sqrt{\pi}}{\Gamma(n)\sqrt{{ n}}}{1\over\sqrt{\vert z\vert}}.
\end{equation}
\end{theorem}

\begin{proof}
For $n\ge1$ and $z=\vert z\vert e^{i \theta}\in\CCma$, we have
\begin{eqnarray*}
\Vert K_{n,z}\Vert_{2,(n)}^2&=&K_{n,z}(z) =\int_{0}^{1}\int_{0}^{1}\frac{(1-y)^{n-1}}
{\Gamma(n)}\frac{(1-x)^{n-1}}{\Gamma(n)}\frac{1}{x\zc+yz}dxdy\cr
&=&\frac{1}{|z|}\int_{0}^{1}\int_{0}^{1}\frac{(1-y)^{n-1}}{\Gamma(n)}\frac{(1-x)^{n-1}}{\Gamma(n)}\frac{(x+y)\cos\theta}{x^{2}+y^{2}+2xy\cos(2\theta)}dxdy,
\end{eqnarray*}
since $\Im(K_{n,z}(z))=0$ (use symmetry in the imaginary part of the integral). Thus using symmetry first and then the change of variables $x=yt$ one obtains
\begin{eqnarray*} K_{n,z}(z) &=&
\frac{2\cos\theta}{\Gamma(n)^2|z|}\int_{0}^{1}\int_{0}^{y}\frac{(1-x)^{n-1}(1-y)^{n-1}(x+y)}{x^{2}+y^{2}+2xy\cos(2\theta)}dxdy\cr
&=&\frac{2\cos\theta}{\Gamma(n)^2|z|}
\int_{0}^{1}\int_{0}^{1}\frac{(1+t)(1-yt)^{n-1}(1-y)^{n-1}}{t^{2}+1+2t\cos(2\theta)}dtdy.
\end{eqnarray*}

Given that, one has $(1+t)(1-yt)^{n-1}\le 2$ if $t,y\in (0,1),$  it follows that the  following upper estimate holds
\begin{eqnarray*}
K_{n,z}(z)
&\le& \frac{4\cos\theta}{\Gamma(n)^2|z|}\int_{0}^{1}\int_{0}^{1}\frac{(1-y)^{n-1}}{t^{2}+1+2t\cos(2\theta)}dtdy\cr&=& \frac{4}{n\Gamma(n)^2|z|}I(\theta)\le \frac{\pi}{n\Gamma(n)^2|z|}.
\end{eqnarray*}
Notice that we have used Lemma \ref{integ} in the last inequality.

For the lower norm estimate now, observe that, for $t,y\in (0,1)$ one has $(1+t)(1-yt)^{n-1}\ge (1-y)^{n-1}$ and therefore
\begin{eqnarray*}
K_{n,z}(z)
&\ge&\frac{2\cos\theta}{\Gamma(n)^2|z|}\int_{0}^{1}\int_{0}^{1}\frac{(1-y)^{2n-2}}{t^{2}+1+2t\cos(2\theta)}dtdy\cr
&=&\frac{2}{\Gamma(n)^2|z|(2n-1)}I(\theta)\ge \frac{1}{\Gamma(n)^2|z|(2n-1)},
\end{eqnarray*}
by Lemma \ref{integ} again.
\end{proof}

\begin{remark}
With regards to the relationship between norms of the kernels note that, for $n=0$, we have
$\Vert K_{0,z}\Vert_{2}=\Vert K_{z}\Vert_{2}=\displaystyle{1\over \sqrt{2\Re z}}$ and then
$$
\|K_{1,z}\|_{2,(1)}\le \sqrt{2\pi} \|K_{0,z}\|_{2}, \quad z\in\C^+.
$$

For $n\ge2$ and $z=\vert z\vert e^{i\theta}$, it follows easily from the equality

$$
K_{n,z}(z)
=\frac{2\cos\theta}{\Gamma(n)^2|z|}
\int_{0}^{1}\int_{0}^{1}\frac{(1+t)(1-yt)^{n-1}(1-y)^{n-1}}{t^{2}+1+2t\cos(2\theta)}dt\ dy
$$
that
$$
K_{n,z}(z)\le {1\over (n-1)^2}K_{n-1,z}(z), \quad z\in\C^+.
$$
\end{remark}

Concerning point estimates of functions in the Hardy-Sobolev spaces, recall that it is well known that for every
$F\in H_2$,
$$
|F(z)|^{2} \leq \frac{\|F\|^{2}_{2}}{4\pi \Re z}, \qquad z\in \C^+,
$$
where the constant $1/4 \pi$ is the best possible, see for example \cite[Lemma 3.2]{Arvan} and \cite[Chapter VI.C]{Duren}.
Similar bounds are available now for functions $F\in H_2^{(n)}$ .

\medskip
\begin{corollary} Let $n\in\NN$ and $F\in\HniiCma$. Then for each $z\in\C^+$
$$
|F(z)|^{2} \leq \frac{\pi \|F\|^{2}_{2,(n)}}{\Gamma(n)^{2}n |z|}.
$$
\end{corollary}

\begin{proof} Let $F\in\HniiCma$ and $z\in\CCma$. By {the} Cauchy-Schwarz inequality and Theorem \ref{cota}, we have
$$
|F(z)|^{2} = |\left( F\mid K_{n,z} \right)_{(n)}|^{2} \leq \|F\|^{2}_{2,(n)} \|K_{n,z}\|^{2}_{2,(n)}
\leq \frac{\pi \|F\|^{2}_{2,(n)}}{\Gamma(n)^{2}n |z|},
$$
and the proof is over.
\end{proof}

\section{Analytic selfmaps of $\C^+$}\label{selfmaps}

Here we point out some properties of analytic selfmaps $\varphi: \C^+\to \C^+$, in preparation for discussing
composition operators on $H_2^{(n)}$ in Section \ref{compositionOp}.
So in particular, in the rest of the paper we deal with selfmaps $\varphi$ which fix the point at infinity and have a finite angular derivative at $\infty $ in the sense of Carath\' eodory; that is,
$\varphi(\infty)=\infty$ and $\varphi'(\infty)<\infty$ where
$$
\varphi(\infty)=\lim_{\vert z\vert\to +\infty}\varphi(z), \qquad
\varphi'(\infty)=\lim_{z\to \infty}{z\over\varphi(z)},
$$
the second limit above being non-tangential (see details about angular derivatives and nontangential limits in \cite[Chapter 4]{Shapiro}). In \cite[Proposition 2.2]{Elliot-Jury}, a version for $\C^+$ of the Julia-Carath\'eodory  theorem is established. That result states that
the following three conditions are equivalent:
\begin{itemize}
\item[(a)] $\varphi(\infty)=\infty$ and $\varphi'(\infty)<\infty$;
\item[(b)]  $\displaystyle{\sup_{z\in \C^+}{\Re z\over \Re(\varphi(z))}}<\infty$;
\item[(c)]  $\displaystyle{\limsup_{z\to \infty}{\Re z\over \Re(\varphi(z))}}<\infty$.
\end{itemize}
Further, if (a), (b) or (c) hold then
\begin{equation}\label{finiteangular}
\varphi'(\infty)=\displaystyle{\sup_{z\in \C^+}{\Re z\over \Re(\varphi(z))}}=
\displaystyle{\limsup_{z\to \infty}{\Re z\over \Re(\varphi(z))}},
\end{equation}

We now show a simple condition for the existence and finiteness of $\varphi^\prime(\infty)$. The proof is inspired in the proof of Julia's inequality. We include it to avoid the lack of completeness.

\begin{proposition} \label{bound} Let $\varphi$ be an analytic selfmap of $\C^+$ such that
\begin{equation}\label{cota2}
\sup_{z\in \C^+} \frac{\vert z\vert }{\vert\varphi (z)\vert}<\infty.
 \end{equation}
Then the map $\varphi$ satisfies (a), (b), (c) prior to (\ref{cociente}). Moreover
$$
\varphi'(\infty)=\sup_{z\in \C^+} \frac{\Re z}{\Re \varphi (z) } \le \sup_{z\in \C^+} \frac{\left|z\right|}{\left|\varphi (z)\right|}.
$$
\end{proposition}

\begin{proof}
Let $\phi$ denote the conformal conjugate $\phi =\gamma ^{-1}\circ \varphi \circ \gamma$ of
$\varphi$ via the Cayley transform $\gamma$ introduced in Section 3. Then, one can easily check that
\begin{equation}\label{DiscHalfPlaneID}
\frac{z}{\varphi(z)}=\frac{1-\phi(\lambda)}{1-\lambda}\frac{1+\lambda}{1+\phi (\lambda)},
\end{equation}
for $\lambda\in \mathbb{D}$ and $z=\gamma (\lambda)\in\C^+$. Set $M=\sup \{ |z|/|\varphi (z)|: z\in \C^+\}$.
One has $\varphi(\infty)=\infty$ since $M<\infty$, and so $\lim _{\lambda\to 1}\phi (\lambda)=1$.
By (\ref{DiscHalfPlaneID}), one can write
$$
\left| \frac{1-\phi (\lambda)}{1-\lambda}\right|\leq M \left| \frac{1+\phi (\lambda)}{1+\lambda}\right|,
\ \mbox{whence}\ \limsup_{\lambda\to 1}\left| \frac{1-\phi (\lambda)}{1-\lambda}\right|\leq M<\infty .
$$

By \cite[Remark 2]{MatacheCAOT}, the consequence of the above condition is the fact that  $\phi' (1)<\infty $, hence
$\varphi ^\prime (\infty )<\infty $ and so (a), (b), (c) hold.

It is clear that the nontangential limit
$\displaystyle\varphi^\prime(\infty)
=\lim_{z\to\infty}\frac{\vert z\vert}{\left|\varphi (z)\right|}$
is less than or equal to $\displaystyle\sup_{z\in \C^+} \frac{\left|z\right|}{\left|\varphi (z)\right|}$, whence
$$
\sup_{z\in \C^+} \frac{\Re z}{\Re \varphi (z) } \le \sup_{z\in \C^+} \frac{\left|z\right|}{\left|\varphi (z)\right|},
$$
and the proof is over.
\end{proof}

\begin{example}\label{cociente}
\begin{itemize}
\item[(a)] Let us consider the rational map
$$
r(z)=\frac{a_{l}z^{l}+\cdots+a_{1}z+a_{0}}{b_{m}z^{m}+\cdots+b_{1}z+b_{0}}
$$
with $a_{l},b_{m}\neq 0$. In \cite[Proposition 16]{Elliot} it is proved that, if $r$ is such that $r(\infty)=\infty$ and $r(\CCma)\subseteq\CCma$ then necessarily
\begin{itemize}
\item[(i)] $l=m+1$,
\item[(ii)] $a_{l}/b_{m}\in\RR$, and in particular, $a_{l}/b_{m}>0$,
\item[(iii)] $\Im(a_{0}/b_{0})\geq 0$.
\end{itemize}
In this case $r'(\infty)=b_ma_{m+1}^{-1}$.

In the particular case when $r$  is a M\"obius selfmap of $\C^+$,
such that $r(\infty)=\infty$ and $r'(\infty)<\infty$, the map  $r$ has the form
$$
r (z)=\alpha z+\beta \qquad \alpha >0\, ,\, \Re \beta \geq 0.
$$
\item[(b)] The same can be said on a slightly  larger class  of functions which are quotients  of linear combinations of fractional powers of $z$ and map the set $\C^+$ into itself.
Such a class of functions is denoted by $QLP(\C^+)$ and consists of maps
$$
\rho(z)=\frac{a_{1}z^{\alpha_{1}}+a_{2}z^{\alpha_{2}}
+\cdots+a_{m}z^{\alpha_{m}}}{b_{1}z^{\beta_{1}}+b_{2}z^{\beta_{2}}
+\cdots+b_{p}z^{\beta_{p}}}, \quad z\in \C^+.
$$
where $\alpha_{i},\beta_{j} > 0$ are not necessarily integers. Without loss of generality, one can  assume that properties $\alpha_{i}>\alpha_{i+1}$ and $\beta_{j}>\beta_{j+1}$ hold, see\cite[Section 5]{Elliot}.  If $\rho \in QLP(\C^+)$ and $\alpha_{1}=\beta_{1}+1$ then $\rho(\infty)=\infty$ and $\rho'(\infty)<\infty$. Moreover
$ \rho'(\infty)=\displaystyle{b_1/a_{1}}$.

In particular, the map
$$
\rho(z) = az+b\sqrt{z}+c, \qquad z\in \C^+,
$$
with $a,b>0$ and $\Re c \geq 0$ has been considered in \cite[Example 2.9]{MatacheHalfPlane}.
\item[(c)] As a last example, let $\varphi$ be given by $\varphi(z)=az+b\log(1+z)$ with $a, b>0$. It is clear that $\varphi(\C^+)\subset \C^+$,  $\varphi(\infty)=\infty$ and $\varphi'(\infty)=a^{-1}$.
\end{itemize}
\end{example}

\begin{remark}\label{noequalcond} Note that the quantities involved in Proposition \ref{bound} are not necessarily equal, and that the finiteness of
$\sup_{z\in \C^+} \vert z\vert \vert\varphi (z)\vert^{-1}$ is not necessary for the existence and finiteness of
$\varphi^\prime(\infty)$ . For example, for $\varphi$ given by $\varphi(z)=z+i$, $z\in \C^+$, one has
$$
\sup_{z\in \C^+} \frac{\Re z}{\Re \varphi (z) }=1, \quad \hbox{and}\quad  \sup_{z\in \C^+} \frac{\left|z\right|}{\left|\varphi (z)\right|}=\infty.
$$
In general, a M\"obius map $\varphi(z)=az+b$ satisfies condition (\ref{cota2})
if and only if $a>0$ and $\Re b>0$.
\end{remark}

We now introduce a stronger condition than (\ref{cota2}) which will be considered in the next section.

\begin{definition} Take $\varphi$ an analytic self map of $\C^+$. We say that $\varphi$ verifies the {\sl $n$-times boundedness condition} ($n$-BC) when

\begin{equation}\label{ConditionExtra}
\sup_{z\in \C^+}\left|  z^k\frac{\varphi ^{(k)}(z)}{\varphi (z)}\right|  <\infty,\quad k=1,2,3, \dots, n,
\end{equation}
for $n\ge 1$.
\end{definition}

\begin{example}\label{exams}
\begin{itemize}
\item[(i)] Recall the rational map $r$ given in Example \ref{cociente} (a)
such that $0\not\in r(i\R)$. Then the map $r$ satisfies the $n$-BC for all $n\ge 1$.
\item[(ii)] A map $\rho \in QLP(\C^+)$ satisfying the conditions in Example \ref{cociente} (b) and such that
$0\not\in\rho(i\R)$ must also satisfy the $n$-BC for all $n\ge1$.
\item[(iii)] Let $\varphi$ be a self-map on $\C^+$ satisfying conditions (\ref{cota2}) and such that
$\Vert \varphi'\Vert_\infty<\infty$. Then property $1$-BC holds for $\varphi$. In particular the map
$\varphi(z)=az+b\log(1+z)$ with $a, b>0$ has the $1$-BC property.

Moreover, for $n\ge 2$,
$$
\varphi^{(n)}(z)= {(-1)^{n-1}b(n-1)!\over (1+z)^{n}}, \quad z\in \C^+,
$$
and we conclude that the map $\varphi$ satisfies the $n$-BC property.
\end{itemize}
\end{example}

\section{Boundedness of composition operators on $H^{(n)}_2$}\label{compositionOp}

An operator $C_{\psi, \varphi }$ of type
$C_{\psi, \varphi }f:=\psi f\circ \varphi$
acting on some function space is called a weighted composition operator. Suppose that such an operator
$C_{\psi, \varphi }$ is bounded on some RKHS $\mathcal H$, with kernel $\mathcal K$,
formed by scalar valued functions on some set $\Omega$.
Then the following extended version of the Caughran--Schwarz equation holds (see \cite[Theorem 5]{MatacheCAOT}):
\begin{equation}\label{CaughranSchwarzExtended}
C^*_{\psi, \varphi }\mathcal K_x=\overline{\psi (x)}\mathcal K_{\varphi (x)},\quad x\in \Omega,
\end{equation}
from which one obtains the bound
\begin{equation}\label{NormEstimateCaughran}
\sup_{x\in \Omega} \vert\overline\psi(x)\vert{\Vert \mathcal K_{\varphi (x)}\Vert \over \| \mathcal K_x\| }
\leq \| C_{\psi,\varphi} \|.
\end{equation}

Here we are interested in operators
$C_{\varphi }:=C_{1, \varphi }$ acting on $H^{(n)}_2$,
with $\varphi$ an analytic selfmap of $\C^+$. Let us recall that, as
an immediate consequence of Littlewood's subordination principle (see for example \cite[p.16]{Shapiro}), {\it every} selfmap $\varphi$ on the disc $\D$ induces a bounded composition operator $C_\varphi$ on the  space
$H_2(\D)$, with norm estimate
$\Vert C_\varphi\Vert \le \sqrt{1+\vert \varphi(0)\vert}\left(\sqrt{1-\vert \varphi(0)\vert}\right)^{-1}$ (\cite{CM}).
Unlike the disc case, {\it not all} analytic selfmaps $\varphi $ of $\C^+$ induce bounded composition operators on
$H_2$, the necessary and sufficient condition for boundedness being (see \cite{Elliot-Jury, Singh})
\begin{equation}\label{BoundRealPart}
\varphi ^\prime (\infty )=
\sup_{z\in \C^+} \frac{\Re z}{\Re \varphi (z) }<\infty.
\end{equation}

The need of condition (\ref{BoundRealPart}) for the continuity of $C_\varphi$ is a consequence of applying estimate
(\ref{NormEstimateCaughran}) (with $\psi\equiv 1$)
to the kernel $K(z,w)=(z+\overline w)^{-1}$ in $H_2(\D)$ (as it was noticed in
\cite[Theorem 2.4]{Singh}).
The proof of the sufficiency was obtained initially in \cite[Corollary 4]{MatacheCAOT}, and, by significantly simpler methods in \cite[Theorem 3.1]{Elliot-Jury}.
Actually the norm of $C_\varphi$ is
$\| C_\varphi \| =r(C_\varphi )=\| C_\varphi \| _e=\sqrt{\varphi ^\prime (\infty )}$,
where $r$ denotes the spectral radius and $\| \quad \| _e$  the essential norm (see \cite[Theorem 3.1, 3.4]{Elliot-Jury}).
In particular,  by (\ref{BoundRealPart}),
it is readily seen which M\"obius selfmaps of $\C^+$ induce bounded composition operators on $H_2$, namely only the maps of type
\begin{equation}\label{Mobius}
\varphi (z)=\alpha z+\beta \qquad z\in \C^+,
\end{equation}
for $\alpha >0$ and $\Re \beta \geq 0$ have that property.

In connection with the above results, the main idea underlying papers
\cite{Elliot-Jury} and \cite{Jury} is to approach the adjoint of a (weighted or unweighted) composition operator by using nonnegative kernels. Next, we state in full generality and prove the boundedness principle which is behind some of the results
in \cite{Jury} and \cite{Elliot-Jury}. For this, recall that given two nonnegative kernels $K_1$ and $K_2$ on the same set
$\Omega$, we write $K_1\leq K_2$ if $K_2-K_1$ is a  {nonnegative kernel} on $\Omega$. Also, for every nonnegative kernel $K$ on $\Omega$ and every selfmap
$\varphi $ on $\Omega$, we denote by $K\circ \varphi $ the mapping
$$
K\circ \varphi(x,y):=K(\varphi (x),\varphi (y)),\qquad x,y\in \Omega.
$$
It is not difficult to check that $K\circ \varphi$ is a nonnegative kernel.

\begin{theorem}\label{TheoremJury'sPrinciple}
Let $\mathcal H$ be a  RKHS  consisting of scalar valued functions on some set
$\Omega$, having nonnegative kernel $\mathcal K$,
and the property that the set of all kernel functions of $\mathcal H$ is a linearly independent set.
Then a weighted composition operator $C_{\psi, \varphi }$ is a bounded operator on
$\mathcal H$ if and only if there is some
$M\geq 0$ such that
\begin{equation}\label{ConditionJury}
M^2{\mathcal K}(x,y) \geq
\overline{\psi (x)}{\psi (y)} {\mathcal K}\circ \varphi (x,y),
\quad x,y\in \Omega.
\end{equation}
Furthermore, the following equality holds:
\begin{equation}\label{Minimum}
\min \{ M\geq 0:  (\ref{ConditionJury})\, \, \mbox{\rm holds}\} =\|C_{\psi, \varphi }\| .
\end{equation}
If  $M^2=\sup \{ |\psi (x)|^2\mathcal K\circ \varphi (x,x)/\mathcal K(x,x): x\in S\}$ is finite and satisfies {\rm (\ref{ConditionJury})}, then  $C_{\psi, \varphi }$ is bounded and     $\|  C_{\psi, \varphi } \| =M$.

Clearly,  $C_{1,\varphi }=C_\varphi $ and so, condition {\rm (\ref{ConditionJury})} looks as follows in this particular case:
\begin{equation}\label{ConditionJuryBis}
\mathcal K\circ \varphi \leq M^2\mathcal K.
\end{equation}
\end{theorem}

\begin{proof}
Let $T$ be the following transform on the space $\mathcal{S}$ spanned by the kernel functions of $\mathcal H$:
\[
T\left(\sum _{j=1}^nc_jk_{x_j}\right):=\sum _{j=1}^nc_j\overline{\psi (x_j)}k_{\varphi (x_j)},\qquad  f=\sum _{j=1}^nc_jk_{x_j}\in \mathcal{S}.
\]
The above (obviously linear) transform is  well defined, since the set of all kernel functions of $\mathcal H$ is linearly independent. It is easy to see that $T$ is bounded since, by condition (\ref{ConditionJury}) one easily gets
\[
\left\| T\left(\sum _{j=1}^nc_jk_{x_j}\right)\right\| \leq M\left\| \sum _{j=1}^nc_jk_{x_j}\right\|, \qquad  f=\sum _{j=1}^nc_jk_{x_j}\in \mathcal{S}.
\]
Thus $T$ extends to $\mathcal H$, the closure of $\mathcal{S}$, and this extension by linearity and boundedness, denoted also by $T$, satisfies
$\| Tf\| \leq M\| f\|$, for $f\in \mathcal H$.

Since $T$ is bounded, so is $T^*$. Further, it is readily seen that
$T^*=C_{\psi, \varphi}$
and so $C_{\psi, \varphi }$ is bounded. Indeed,
\[
\langle T^*k_x, k_y\rangle =\psi (y)k_{ x}(\varphi (y))=\langle C_{\psi ,\varphi }k_x, k_y\rangle, \quad x,y\in\Omega.
\]
Note also that $\| C_{\psi ,\varphi }\| \leq M$.

Conversely, assume $\| C_{\psi ,\varphi }\| < \infty $. Then writing
\[
\left\| C_{\psi, \varphi }\left(\sum _{j=1}^nc_jk_{x_j}\right)\right\| ^2\leq \| C_{ \psi, \varphi }\| ^2\left\| \sum _{j=1}^nc_jk_{x_j}\right\| ^2,\quad  f=\sum _{j=1}^nc_jk_{x_j}\in \mathcal{S},
\]
and using (\ref{CaughranSchwarzExtended}), one gets that $M=\| C_{\psi, \varphi }\| $ satisfies condition
(\ref{ConditionJury}). Hence, if that condition holds for some $M>0$, then the least value of $M>0$ for which it holds is $M=\| C_{\psi, \varphi }\| $.

Finally, combining (\ref{CaughranSchwarzExtended}) with
$\sup \{ \| C^*_{\psi, \varphi }k_x\| /\| k_x\|  :x\in S\} \leq \| C_{\psi, \varphi }\|$
one gets that, if  $M^2
=\sup \{ |\psi (x)|^2\mathcal K\circ \varphi (x,x)/\mathcal K(x,x): x\in S\}$ is finite and satisfies {\rm (\ref{ConditionJury})} then $T_{\psi, \varphi }$ is bounded and $\|  C_{\psi, \varphi } \| =M$.
\end{proof}

Straightforward application of relation (\ref{ConditionJuryBis}) in the above theorem gives us the following characterization, in the case of the Hardy-Sobolev space $H_2^{(n)}$. We use the kernel of $H_2^{(n)}$ given in Theorem \ref{RKHS}.

\begin{corollary}\label{character(n)} Let $\varphi$ be an analytic selfmap of $\C^+.$ The composition operator $C_\varphi$ is {a} bounded operator on $H_2^{(n)}$ for $n\ge 1$ if and only if there exists $M >0$ such that the kernel $L_n$ given, for $(z,w)\in\C^+\times\C^+$,  by
\begin{eqnarray*}
&\quad&L_n(z, w):=\cr
&\ &\int _0^1\int _0^1(1-t)^{n-1}(1-s)^{n-1}
\left(\frac{M^2}{zt+s\overline{w}}-\frac{1}{\varphi(z)t+s\overline{\varphi(w)}}\right)ds\  dt,
\end{eqnarray*}
is nonnegative.
\end{corollary}

\medskip
\begin{remark}
\normalfont
Theorem \ref{TheoremJury'sPrinciple} extends  \cite[Theorem 3.1]{Elliot-Jury}  concerning composition operators on $H_2$, where the main idea of the proof is to show that the function $L_0$ on $\C^+\times \C^+$ defined by
\[
L_0(z, w):={\varphi ^\prime (\infty )\over z+\overline{w}}-{1\over\varphi(z)+\overline{\varphi(w)}}, \quad z, w\in \C^+,
\]
is a nonnegative kernel of $H_2$. To do this, note that  $\varphi (z)-(1/\varphi ^\prime (\infty ))z$ is an analytic selfmap
of $\C^+$  if (\ref{BoundRealPart}) holds, and then the kernel
$$
\frac{\varphi (z)-(1/\varphi ^\prime (\infty ))z+\overline{(\varphi (w)-(1/\varphi ^\prime (\infty ))w}}{z+\overline{w}},\qquad z,w\in \C ^+,
$$
is a nonnegative kernel on $\C ^+$ by the Nevanlina--Pick interpolation theorem (\cite[Theorem 2.2]{Garnett}).
Since the product of two nonnegative kernels on the same set is also a nonnegative kernel on that set, it follows that
$$
\frac{1}{\varphi (z)+\overline{\varphi (w)}}\left(\frac{\varphi (z)+\overline{\varphi (w)}}{z+\overline{w}}-\frac{1}{\varphi ^\prime (\infty )}\right)=\left(\frac{1}{z+\overline{w}}-\frac{1}{\varphi ^\prime (\infty )}\frac{1}{\varphi (z)+\overline{\varphi (w)}}\right)
$$
is a nonegative kernel on $\C ^+$, and the proof of  \cite[Theorem 3.1]{Elliot-Jury}  is over.

With regards to Corollary \ref{character(n)}, since the kernel $L_n$ is an integral version of kernel $L_0$, one may expect that some similar ideas could work --under condition (\ref{cota2}) at least, on account of Proposition \ref{PropositionNormFinite} below--
for suitable $M>0$.
However the authors have not been able to prove or find  an integral version of the Nevanlina-Pick interpolation theorem to conclude the nonnegativity of this kernel.

In fact, the characterization of the boundeness of composition operators $C_\varphi$ on the Hardy-Sobolev space $H^{(n)}_2$
does not seem to be simple. In the sequel, we present some partial results in this direction.
\end{remark}

\medskip
Estimate (\ref{NormEstimateCaughran}) together with
(\ref{inequ}) and Proposition \ref{bound} give us the following result.

\begin{proposition}\label{PropositionNormFinite}
If $C_\varphi $ is a bounded composition operator on $H^{(n)}_2$ for some $n\ge 1$,
then the map $\varphi$ satisfies condition (\ref{cota2}).
Hence,  $\varphi (\infty )=\infty$ and $\varphi' (\infty)<\infty$ and  $C_\varphi$ is a bounded operator on $H_2$.
\end{proposition}

\begin{remark}\label{Remark} The above proposition tells us that the semigroup of analytic selfmaps of $\C ^+$  inducing  bounded composition operators  on $H^{(n)}_2$ is a (proper, as we will see in Corollary \ref{CorollaryMoebius})
sub--semigroup of the semigroup of analytic selfmaps of $\C ^+$  inducing bounded composition operators  on ${H}_2$.
\end{remark}

Now, we give some sufficient conditions to conclude that the composition operator $C_\varphi$ on ${H^{(n)} _2}$ is bounded on  ${H^{(n)} _2}$. The following theorem provides  a point of view  different from that contained by Theorem \ref{TheoremJury'sPrinciple}.

\begin{theorem}\label{PropositionBoundednessSufficient}
Let $\varphi $ be an analytic selfmap of $\C ^+ $ satisfying condition {\rm (\ref{BoundRealPart})} and the $n$-BC property {\rm (\ref{ConditionExtra})}; that is,
$$
\sup_{z\in \C^+} \frac{\Re z}{\Re \varphi (z) }<\infty \quad \hbox{ and }\
\sup_{z\in \C^+}\left|  z^k\frac{\varphi ^{(k)}(z)}{\varphi (z)}\right|  <\infty,\ k=1,2,3, \dots, n.
$$
Then the map $\varphi $ induces a bounded composition operator $C_\varphi$ on ${H^{(n)} _2}$.
\end{theorem}

\begin{proof}
Since ${H^{(n)} _2} $ is a RKHS, $C_\varphi $ is a bounded composition operator on ${H^{(n)} _2} $  if and only if $C_\varphi {H^{(n)} _2} \subseteq {H^{(n)} _2} $ which follows by a straightforward application of the Closed Graph Principle. Thus, if $f$ is an arbitrary function in ${H^{(n)} _2} $, showing that $f\circ \varphi \in {H^{(n)} _2} $ will end the current proof. We will prove the aforementioned fact in the following. First note that the $n$--th derivative of the composite of two  analytic functions has the following form:
\begin{equation}\label{ChainRuleOrderN}
(f\circ \varphi )^{(n)}=\sum_{k=1}^nc_k(f^{(k)}\circ \varphi )(\varphi ^\prime )^{m_{k_1}}(\varphi ^{(2)} )^{m_{k_2}}\dots (\varphi ^{(n)} )^{m_{k_n}}
\end{equation}
where $c_1,\dots ,c_n$ and $m_{i_j}$ are nonnegative integers for all $i,j=1,2,\dots , n$ and satisfy the conditions
\begin{equation}\label{ChainRuleCondo_01}
1\cdot m_{k_1}+2\cdot m_{k_2}+\dots +n\cdot m_{k_n}=n,\quad k=1,2,\dots ,n,
\end{equation}
and
\begin{equation}\label{ChainRuleCondo_02}
 m_{k_1}+ m_{k_2}+\dots + m_{k_n}=k,\quad k=1,2,\dots ,n.
\end{equation}

The above description of the $n$--th derivative of the composition of two  analytic functions can be easily checked by induction, (alternatively one can use an exact description of the coefficients $c_k$ known as the F\` aa di Bruno formula). Our goal is to show that $\| z^n(f\circ \varphi )^{(n)}(z)\|_{2}<\infty $. To that aim, note that, by (\ref{ChainRuleOrderN})--(\ref{ChainRuleCondo_02}), one can write
\begin{eqnarray*}
&\quad&\| z^n(f\circ \varphi )^{(n)}(z)\|_2\cr&\quad& \leq
\sum_{k=1}^nc_k\| z^n(f^{(k)}\circ \varphi )(z)(\varphi ^\prime (z))^{m_{k_1}}(\varphi ^{(2)} (z))^{m_{k_2}}
\dots (\varphi ^{(n)} (z))^{m_{k_n}}\| _2\cr
&\quad&=
\sum_{k=1}^nc_k\| (f\circ \varphi )^{(k)}(z)\left(\frac{z\varphi ^\prime (z)}{\varphi (z)}\right)^{m_{k_1}}\dots
\left(\frac{z^n\varphi ^{(n)} (z)}{\varphi (z)}\right)^{m_{k_n}}(\varphi (z))^k\|_2\cr
&\quad&\leq
\sum_{k=1}^n c_k\prod _{j=1}^n\left\|\frac{z^j\varphi ^{(j)} (z)}{\varphi (z)}\right\| _\infty ^{m_{k_j}}
\| (\varphi (z))^k(f\circ \varphi )^{(k)}(z)\|_2<\infty.
\end{eqnarray*}
The quantity above is finite by conditions {\rm (\ref{BoundRealPart})} and (\ref{ConditionExtra}).
Indeed, given that $z^kf^{(k)}(z)\in H_2$ for all $k=1,2 ,\dots ,n$ and $C_\varphi $ is a bounded operator on $H_2$,
it follows that
\[
\| C_\varphi (z^kf^{(k)}(z))\| _2=\| (\varphi (z))^k(f\circ \varphi )^{(k)}(z)\|_2<\infty, \qquad k=1,2 ,\dots ,n.
\]
The proof is over.
\end{proof}

An immediate consequence of the above theorem is the following.

\begin{corollary}\label{CorollaryMoebius}
The M\"obius maps $\mu$ inducing bounded composition operators on $H^{(n)}_2$ for $n=1,2,3, \dots$ are exactly those of the form
$$
\mu(z)=az+b, \quad z\in\C^+,
$$
with $a\in\R$ and $b\in \C$ such that $a, \Re b>0$.
\end{corollary}
\begin{proof}
Take $n\in\N$. By Proposition \ref{PropositionNormFinite} every M\"obius map $\mu$ such that $C_\mu$ is bounded on $H^{(n)}_2$ must satisfy condition \eqref{cota2}, and then is $\mu(z)=az+b$ for all $z\in\C^+$ with $a>0$, $\Re b>0$; see Remark \ref{noequalcond}.
Conversely, every M\"obius map $\mu$ of type $\mu(z) = az+ b$ on $\C^+$ with $a >0$ and $\Re b>0$ satisfies condition
\eqref{BoundRealPart} and the $n$-(BC) property \eqref{ConditionExtra}. Thus by
Theorem \eqref{PropositionBoundednessSufficient} $C_\mu$ is bounded on $H^{(n)}_2$.
\end{proof}

In the following we consider several concrete analytic selmaps $\varphi$ of $\C^+$, and discuss the boundeness of the corresponding composition operators $C_\ffi$.

\begin{example} Note that the contrapositive of Proposition \ref{bound} and
Proposition \ref{PropositionNormFinite} state that, if $\ffi$ does not induce a bounded composition operator $C_{\ffi}$ on $H_2$ (i.e. $\varphi'(\infty)=\infty$), then $C_\ffi$ cannot be bounded on $H_2^{(n)}$. Then one can use \cite[Corollary 2.2 and Example 2.4]{MatacheHalfPlane}  to deduce that neither a
bounded mapping $\ffi$ nor the map $\ffi(z)=\sqrt{z}$ can induce a bounded composition operator on
$\HniiCma$.
\end{example}

\begin{example}\label{concret}
Let $r$ be a rational map  like in Example \ref{exams} (a). Then the operator $C_{r}$ is bounded on $H_2$ (\cite[Corollary 17]{Elliot}). By Theorem \ref{PropositionBoundednessSufficient}, $C_{r}$ is also bounded on $H_2^{{(n)}}$ for  $n\ge 1$.
Similarly, the theorem can be applied to maps $\rho$ of $QLP(\C^+)$ in Example \ref{exams} (b). In particular the map
$$
\ffi(z) = az+b\sqrt{z}+c, \qquad z\in \C^+,
$$
with $a,b>0$ and $\Re c \geq 0$, induces a bounded composition operator on $H_2^{{(n)}}$ for  $n\ge 1$.
The case $n=0$ was proved in \cite[Example 2.9]{MatacheHalfPlane}.

As a last example, let  $\varphi$ be given by $\varphi(z)=az+b\log(1+z)$ with $a, b>0$ as in Example \ref{exams} (c).
Then the operator $C_{\varphi}$ is bounded on $H_2^{{(n)}}$ for  $n\ge 1$.
\end{example}

\section{Concluding remarks}

Let $\varphi\colon\C^+\to\C^+$ be an analytic function and let $C_\varphi$ be the composition operator of symbol $\varphi$.
For $n\in\N$, we have shown that that the finiteness condition \eqref{cota2}, i. e.,
$$
\sup_{w\in \C^+} \frac{\vert w\vert }{\vert\varphi (w)\vert}<\infty,
$$
is necessary for $C_\varphi$ to be bounded on $H_2^{(n)}$.

\begin{question} Let $n$ be a natural number. Is \eqref{cota2} also sufficient in order to get $C_\varphi$ bounded
on $H_2^{(n)}$ ?
\end{question}

\begin{remark}
Since condition (\ref{cota2}) is independent of $n$, if the above question
has a positive answer then certainly a map $\varphi $ simultaneously induces bounded composition operators on all spaces $H_2^{(n)}$,  $n\in \N$. Nonetheless, since spaces $H_2^{(n)}$ involve derivatives significantly, it sounds natural --even more if the answer to Question 8.1 is negative-- to raise some questions about
$\varphi$ and its derivatives.
\end{remark}

On account of Theorem
\ref{PropositionBoundednessSufficient}, where condition (\ref{ConditionExtra}),
$$
\sup_{z\in \C^+}\left|  z^k\frac{\varphi ^{(k)}(z)}{\varphi (z)}\right|  <\infty,\quad k=1,2,3, \dots, n,
$$
is considered, we pose also the following.

\begin{question}Let $n$ be a natural number. Does an analytic selfmap of $\C^+$ induce a bounded composition operator on $H_2^{(n)}$ if and only if {\rm (\ref{ConditionExtra})} holds true ? If the answer were positive then condition
{\rm (\ref{ConditionExtra})} would imply condition (\ref{BoundRealPart}), which looks not very likely. So it is possible that we have to include condition (\ref{BoundRealPart}) in the question.
\end{question}

\begin{remark}
Finally, it seems natural giving connections between composition operators acting on spaces
$H_2^{(n)}$ and composition operators  on weighted subspaces of $H_{2}(\D)$, using
Proposition \ref{PropositionUnitDisc} for $n=1$ and what it is noted after that proposition, for $n=2$.
\end{remark}
\subsection*{\it Acknowledgments.} Authors thank  the anonymous referee and  Eva Gallardo-Guti\'errez for their advice,  comments and references which lead to the current, improved version, of this paper.

\end{document}